\documentclass[a4paper,12pt]{amsart}
\usepackage[utf8]{inputenc}

\usepackage{url}
\usepackage[margin=1in]{geometry}  
\usepackage{epsfig}
\usepackage{color}
\usepackage{graphicx}              
\usepackage{amsmath}
\usepackage{amscd}               
\usepackage{amsfonts}
\usepackage{amssymb}             
\usepackage{amsthm}                
\usepackage{epsfig}
\usepackage{mathrsfs}

\newtheorem{main_theorem}{Theorem}
\newtheorem{theorem}{Theorem}

\newtheorem*{theorem*}{Theorem}
\newtheorem{lemma}[theorem]{Lemma}
\newtheorem*{lemma*}{Lemma}
\newtheorem{proposition}[theorem]{Proposition}
\newtheorem*{proposition*}{Proposition}
\theoremstyle{definition}
\newtheorem*{definition}{Definition}

\theoremstyle{remark}

\newtheorem*{remark}{Remark}

\newcommand{\good}{\operatorname{good}}
\newcommand{\LCM}{\operatorname{LCM}}

\newcommand{\Prob}{\mathbf{P}} 
\newcommand{\one}{\mathbf{1}}

\newcommand{\zed}{\mathbb{Z}}

\newcommand{\bN}{\mathbb{N}}

\newcommand{\cN}{\mathscr{N}}		 
\newcommand{\cM}{\mathscr{M}}

\title{Solution of the minimum modulus problem for covering systems}
\author{Bob Hough}
\address{Mathematical Institute, University of Oxford, 
Radcliffe Observatory Quarter, Woodstock Road, Oxford, OX2 6GG, UK.}
\email{hough@maths.ox.ac.uk}
\thanks{This research was begun while the author was a postdoctoral
research fellow at Department of Pure Maths and Math Stats, Cambridge, and
completed while a postdoctoral research fellow at the Mathematical Institute,
Oxford.  He is grateful for financial support from ERC Research Grant
279438, Approximate Algebraic Structure and Applications.}


\begin{document}
\begin{abstract}
We answer a question of Erd\H{o}s by showing that the least modulus
of a distinct covering system is at most $ 10^{16}$.
\end{abstract}

\maketitle

\section{Introduction}
In 1934 Romanoff   proved that the numbers of form a prime plus a power
of two have positive lower density.  Writing to Erd\H{o}s, he asked whether
there
exists an arithmetic progression of odd numbers none of whose members is of
this form. Erd\H{o}s's positive answer to this
question introduced the
notion of a
\emph{distinct covering system of congruences}, which is a finite collection of
congruences
\[
 a_i \bmod m_i, \qquad 1 < m_1 < m_2 < ... < m_k
\]
such that every integer satisfies at least one of them.  His  paper
\cite{E50} gives the example
\[
 0\bmod 2,\quad 0 \bmod 3, \quad 1\bmod 4,\quad 3 \bmod 8, \quad 7 \bmod 12,
\quad 23 \bmod 24.
\]

Erd\H{o}s posed a number of problems concerning covering systems, of which two
in particular are well known. From \cite{E50}, the minimum modulus
problem asks whether there exist distinct covering systems for which the least
modulus is arbitrarily large.  With Selfridge, Erd\H{o}s
asked if there exists a distinct covering system with
all moduli odd. These two questions appear  frequently in Erd\H{o}s' collections
of open problems \cite{E57}, \cite{E63}, \cite{E69}, \cite{E73},
\cite{EG80}.  See also \cite{G94}.

Following Erd\H{o}s'  paper, a number of covering systems have been
exhibited
with increasing minimum modulus \cite{C68}, \cite{K71}, \cite{C71}, \cite{M84},
\cite{G09}, 
with the current record of 40 due to Nielsen \cite{N09}. In
\cite{N09}, Nielsen suggests for the first time that the answer to the minimum
modulus problem may be negative.  We confirm this conjecture.

\begin{main_theorem}\label{main_result}
 The least modulus of a distinct covering system is at most $
10^{16}$.
\end{main_theorem}
To obtain the   bound of $10^{16}$ we use some simple
numerical calculations performed in Pari/GP \cite{PariGP}, together with a
standard explicit estimate for the counting function of primes.
For the reader
interested only in the qualitative statement that the minimum modulus has a
uniform upper bound, our presentation is  self-contained.

In the spirit of the odd modulus problem, Theorem \ref{main_result} immediately
implies that any covering
system
contains a modulus divisible by one of an initial segment of primes. We may
return to give a stonger quantitative statement of this type at a
later time.

Prior to our work, the main theoretical progress on the minimum modulus
problem was made recently by Filaseta, Ford, Konyagin, Pomerance and Yu
\cite{FFKPY07}, who showed, among other results, a lower bound for the sum
of the reciprocals
of the moduli of a covering system that grows with the minimum modulus. We build
upon their work. In particular, we use an inductive scheme
in which we filter the moduli of the congruences according to the
size of their prime factors, so that we first consider the
subset of congruences all of whose
prime factors are below an initial threshold, and then increase the threshold
in stages.  The paper \cite{FFKPY07} roughly makes the first stage of this
argument.

A detailed overview of our argument is given in the next section, but we mention
here that our proof follows the probabilistic
method in the sense that we give a positive lower bound for the density of
integers left uncovered by any distinct system of congruences for which the
minimum modulus is sufficiently large. The
Lov\'{a}sz Local Lemma plays a
crucial r\^{o}le.  The suitability of the Local Lemma for
estimating the density of the uncovered set at each stage of the argument relies
upon a certain regularity of the uncovered set from the previous stage, and this
regularity we are able to guarantee by applying the Local Lemma a second time,
in a relative form.

\subsubsection*{Notation} Throughout we denote $\omega(n)$ the number of
distinct prime factors of
natural number $n$.
\subsubsection*{Acknowledgements}
The author is grateful to Ben Green, who read an early version of this
paper and made a number of suggestions that dramatically improved the
structure and readability. The author is also grateful to Pace Nielsen,
Kevin Ford and Michael Filaseta for detailed comments, and to
K. Soundararajan and Persi Diaconis, from whom he learned many of the methods
applied here. An anonymous referee pointed out a numerical improvement to the
parameters
which lowered the final bound.

\section{Overview}
We begin by giving a reasonably detailed overview of the argument.  In this
summary we will consider only 
congruence systems all of whose moduli are square free. Treating the case of
general moduli involves a minor complication, which we address in the next
section.

Let $M>1$ and let
\[
 \mathscr{M} \subset \{m \in \bN: m \text{ square free, } m > M\}
\]
be a finite set of moduli.  We assume that for each
$m
\in \mathscr{M}$  a residue class $a_m \bmod m$ has been given. For $M$
sufficiently large, we argue that for any $\cM$, and for any assignment of the
$a_m$, we can give a positive lower
bound for the density of solutions to the system
of (non)-congruences
\[
R = \{ z \in \zed:\;\;  \forall m \in \mathscr{M},\;  z \not \equiv
a_m \bmod m\}.
\]
The bound will, of course, depend upon $\cM$.

We estimate the density of $R$ in stages, so we introduce a sequence of
thresholds $ 1=P_{-1}< P_0 < P_1 < ...$ with $P_i \to \infty$.  For the purpose
of this
summary we assume that $P_0$ is sufficiently small so that $\prod_{p \leq P_0} p
< M$, although to get a better bound for $M$, we will in practice  choose
$P_0$ to be somewhat larger. 
Let $1 = Q_{-1}$, $Q_0$, $Q_1$, ... be such that
\[
 Q_i = \prod_{p \leq P_i} p,\qquad i \geq 0.
\]
We say that a number $n$ is $P_i$-smooth if $n|Q_i$.  
Let $\cM_0, \cM_1,  ...$ be given by
\[
 \cM_i = \{m\in \cM: m|Q_i\}, \qquad i \geq 0,
\]
that is, $\cM_i$ is the set of $P_i$-smooth moduli in $\cM$. In particular, by
our assumption on $P_0$ we have that $\cM_0$ is empty.   For this reason we set
$R_0 =R_{-1} =  \zed$, and consider the sequence of unsifted sets $R_{0}\supset
R_1
\supset R_2 \supset  ...$
\[
 R_{i} = \bigcap_{m \in \cM_i} \{z \in \zed: z \not \equiv
a_m \bmod m\}, \qquad i \geq 1.
\]
Since the sets $\cM_i$ grow to exhaust $\cM$,   we eventually
have $R = R_i$, and so it will suffice to prove that the density of $R_i$ is
non-zero for each $i$.  This lower bound we will give uniformly for all
 congruence systems with minimum modulus greater than $M$.

We may view $R_i$ as a subset of $\zed/Q_i\zed$.  Thinking of
$\zed/Q_{i+1}\zed$ as fibred over $\zed/Q_i\zed$, we then have that
$R_{i+1}$ is contained in fibres over $R_i$ and we may estimate the density of
$R_{i+1}$ by estimating its density in individual fibres over $R_i$. In fact,
we only consider some `good' fibres over a `well-distributed' subset  of $R_i$. 
Thus  we do not
actually estimate the density of $R_{i+1}$, but rather that of a somewhat
smaller set.  Also, rather than explicitly estimate the density of the smaller
set, we will check that the smaller set is non-empty and then estimate some
statistics related to it. 

Let $i \geq 0$ and let $r \in R_i \bmod Q_i$.  By definition, $r$ has survived
sieving by all of
the congruences to moduli dividing $Q_i$, so that the fraction of the fibre
$(r \bmod Q_i)$ that survives into $R_{i+1}$ is determined by congruence
conditions
to moduli in $\cM_{i+1}\setminus \cM_i$.  Each such modulus $m$ has a unique
factorization as $m = m_0 n$ with $m_0 |Q_i$ and $n$ composed of primes in
the interval $(P_i, P_{i+1}]$.  We call the collection of such $n$ the set of
`new factors'
\[\forall i \geq 0, \qquad \cN_{i+1} = \{n \in \bN: n > 1, n \text{
square free}, p|n \Rightarrow p
\in (P_i, P_{i+1}]\}.\] This set will play a very important r\^{o}le in what
follows.

Given $r \in R_i \bmod Q_i$,   $a_{m_0n} \bmod m_0n$ intersects $(r \bmod Q_i)$
 if and only if $a_{m_0 n} \equiv r \bmod m_0$.  If this
condition is met, the effect within the fibre is determined only by $a_{m_0 n}
\bmod n$.  For this reason, we group together the congruence conditions
according to common $r$ and $n$: for each $r \in \zed/Q_i\zed$ and each $n
\in \cN_{i+1}$ we set
\[
 A_{ n, r} = (r \bmod 
Q_i) \cap \bigcup_{m_0 | Q_i, m_0 n \in \cM }
(a_{m_0n} \bmod m_0n).
\]
We then have 
\[
\forall i \geq 0, \qquad (r \bmod Q_i) \cap  R_{i+1}   = (r \bmod Q_i) \cap
\bigcap_{n\in
\cN_{i+1}} A_{ n, r}^c,
\]
with the interpretation that $R_{i+1}$ within $(r \bmod Q_i)$ results from
sieving $(r \bmod Q_i)$ by sets of residues to moduli in $\cN_{i+1}$.

When $n_1, n_2 \in \cN_{i+1}$ are coprime, sieving by the sets $A_{n_1, r}$ and
$A_{n_2,
r}$ are independent events, by the Chinese Remainder Theorem.  If
all of
the sets $\{A_{n,r}\}_{n \in \cN_{i+1}}$ were jointly independent, then the
density of the fibre $r \bmod Q_{i}$ surviving into $R_{i+1}$  would
be
\begin{equation*}
 \prod_{n \in \cN_{i+1}} \left(1 - \frac{|A_{n,r}\bmod nQ_i|}{n}\right) \doteq
\exp\left(-\sum_{n \in \cN_{i+1}} \frac{|A_{n,r}\bmod nQ_i|}{n}\right).
\end{equation*}
For a given $n$ we can bound the average size of $|A_{n,r}\bmod nQ_i|$ averaged
over
$r \bmod Q_i$:
\begin{align*}
 \frac{1}{Q_i} \sum_{r \bmod Q_i} |A_{n,r}\bmod nQ_i| &\leq \frac{1}{Q_i}
\sum_{r \bmod
Q_i}\sum_{m_0|Q_i} \one\{a_{m_0n} \equiv r \bmod m_0\}\\
& = \frac{1}{Q_i}\sum_{m_0|Q_i} \sum_{r \bmod Q_i} \one\{r \equiv a_{m_0n}
\bmod m_0\}\\& = \frac{1}{Q_i} \sum_{m_0|Q_i} \frac{Q_i}{m_0} =
\prod_{p|Q_i}\left(1 + \frac{1}{p}\right) = (\log P_i)^{1+o(1)}.
\end{align*}
With the belief that the typical set $A_{n,r}$ has size $\approx \log
P_i$, then since 
\[
 \sum_{n \in \cN_{i+1}}\frac{1}{n} = -1 + \prod_{P_i < p \leq P_{i+1}} \left(1
+ \frac{1}{p}\right) \approx \frac{\log P_{i+1}}{\log P_i}
\]
we might hope that the typical fibre above $R_i$ has density $P_{i+1}^{-O(1)}$.
Thus far our
reasoning in the case $i = 0$ roughly follows the treatment of \cite{FFKPY07},
but now we diverge. 

One difficulty with this heuristic account 
is that for generic $n_1, n_2 \in \cN_{i+1}$ it is not generally true that
$(n_1, n_2) = 1$, so that the congruences in $A_{n_1, r}$ and $A_{n_2, r}$ are
not independent. 
To clarify the situation, we may imagine the numbers in the set $\cN_{i+1}$
as being split into two types.  Within the collection of numbers that are
composed
of `few' prime factors, it is generally true that most pairs of numbers in the
set are co-prime.  Meanwhile, the numbers composed of many prime factors are
large and sparse, and thus may be expected to not  contribute
significantly to the sieve.  This reasoning  makes it plausible that the
Lov\'{a}sz
Local Lemma can be used to handle
the mild dependence that results from sieving by the moduli in $\cN_{i+1}$. In
practice, rather than split the moduli into two groups, in applying the Local
Lemma we are naturally  led to make a smoother decomposition, which assigns
to
each modulus a weight according to its number of prime factors.

Unfortunately, it will not generally be true that the
Local Lemma applies to estimate the density of a given fibre, but rather only
that it applies on a certain subset $R_i^* \subset R_i$ of `good' fibres
on which
the distribution of the
sizes $\{|A_{n,r}\bmod nQ_i|\}_{n \in \cN_{i+1}}$ is under control.
Roughly what is needed for a fibre to be good is that a bound in dilations
should hold at each
prime $p \in (P_i, P_{i+1}]$, 
\begin{equation}\label{good_property}
\sum_{n \in \cN_{i+1}, p|n} \frac{
|A_{n,r} \bmod n Q_{i}|}{n}\ll 1.
\end{equation}
Such a bound controls the dependence among the sets $\{A_{n,r}\}_{n \in
\cN_{i+1}}$.  We  give a more precise definition of good fibres in the next
section.

In order to demonstrate that a reasonable number of fibres are good we wish to
understand the distribution of values of $|A_{n,r}\bmod n Q_i|$ for varying $r$
and $n$.   Recall that we gained a heuristic understanding 
of the typical behavior
of $|A_{n,r}\bmod nQ_i|$ by taking the average over $\zed/Q_i\zed$. 
Similarly, we control the distribution of $|A_{n,r} \bmod nQ_i|$ as $r$ varies
in subsets $S_i$ of $R_i$ by bounding the moments
\[
 \frac{1}{|S_i \bmod Q_i|} \sum_{r \in S_i \bmod Q_i} |A_{n,r} \bmod nQ_i|^k,
\qquad k = 1, 2, 3, ....
\]
It transpires that these moments are controlled by statistics 
\[
 \sum_{m|Q_i} \ell_k(m) \max_{b \bmod m} \frac{|S_i \cap (b
\bmod m) \bmod Q_i|}{|S_i \bmod Q_i|}, \qquad k = 1, 2, 3, ....
\]
that measure the
bias in the set $S_i$.  Here $\ell_k(m)$ is a weight, equal to
$(2^k-1)^{\omega(m)}$ in the case that $m$ is square free. When $i = 0$ it will
not be necessary to consider subsets of $R_0 = \zed/Q_0\zed$, since the 
statistics taken over $R_0$ are unbiased, equal to
\begin{equation}\label{unbiased}
 \sum_{m|Q_0} \frac{\ell_k(m)}{m} = \prod_{p<P_0} \left(1 +
\frac{2^{k}-1}{p}\right) \approx (\log P_0)^{2^k-1},
\end{equation}
a rate of growth which will be acceptable for us.  When $i > 0$,
however, the set $R_i$
will typically be  small and irregular as compared to $\zed/Q_i\zed$, so
that our
argument requires searching for good fibres
$R_i^*$ only within a subset $S_i\subset R_i$ chosen to have
statistics that approximate (\ref{unbiased}).

The above discussion suggests that there is a second convenient
notion of a good fibre, which is that $(r \bmod Q_i)$ is `well-distributed' if
for each $n \in \cN_{i+1}$,
\begin{equation}\label{well_distributed_property}
 \max_{b \bmod n} |R_{i+1}\cap (b \bmod n) \cap (r \bmod Q_i) \bmod Q_{i+1}|
\approx \frac{1}{n} |R_{i+1} \cap (r \bmod Q_i) \bmod Q_{i+1}|.
\end{equation}
Thus in a well-distributed fibre $(r \bmod Q_i)$, for each modulus $n \in
\cN_{i+1}$, any residue class modulo $n$ is allowed to hold at most slightly
more than its share of the set $R_{i+1}$.
A pleasant feature of our argument is that a relative form of the Lov\'{a}sz
Local Lemma  guarantees that good fibres in the sense of
(\ref{good_property}) are automatically well-distributed in the sense of
(\ref{well_distributed_property}), so that with respect to the moduli in
$\cN_{i+1}$ composed of large prime factors, a reasonable choice for the set
$S_{i+1}$ is the union of good fibres
from the previous stage, $S_{i+1} = R_{i}^* \cap R_{i+1}$.  

The choice of $S_{i+1} = R_{i}^* \cap R_{i+1}$ ensures that $S_{i+1}$ is
well-distributed
to the moduli in $\cN_{i+1}$ that have only large prime factors, but $R_{i}^*
\cap
R_{i+1} \subset S_i$ may have become poorly distributed as compared to $S_i$
with respect to moduli having smaller
prime factors as a result of variable  sieving in the fibres above
$R_{i}^*$. We balance this effect by reweighting $R_{i}^* \cap R_{i+1}$ with a
measure $\mu_{i+1}$ on $\zed/Q_{i+1}\zed$, with respect to which each fibre
over $R_{i}^*$ has equal weight.  Thus at stage $i+1 \geq 1$ we will in fact
consider
the bias statistics
\[
 \beta_k^k(i+1) = \sum_{m |Q_{i+1}} \ell_k(m) \max_{b \bmod m}\frac{\mu_{i+1}
(R_{i}^* \cap
R_{i+1} \cap (b
\bmod m))}{\mu_{i+1}(R_{i}^* \cap R_{i+1})}.
\]
In general we will be able to show that these statistics approximate the
unbiased statistics (\ref{unbiased}) to within an
error determined only in terms of the quality of well-distribution
(\ref{well_distributed_property}) and
the fractions of fibres that are good from previous
stages.

To summarize, at
stage 0 we do no sieving so that, with a uniform measure, the bias statistics
are under control.  This allows us to say that many fibres over $R_0 =
\zed/Q_0\zed$ are good, and thus, that the bias statistics at stage 1 do not
grow too rapidly. The argument then iterates, with the possibility of continuing
iteration for arbitrarily large values of the parameters $P_i$
depending upon growth of the statistics $\beta(i)$ as compared with 
growth of the $P_i$.  The proof is completed by making this comparison for an
explicit choice of parameters.

\section{The complete argument}
We turn to the technical details of the argument.  As we now treat
congruences to general moduli, we briefly recall some notions from the
previous section,
pointing out the minor variation from the square free case.

As above, $M > 0$ is our upper bound for  the minimum modulus of a covering
system, and \[\cM 
\subset \{m \in \zed, \; m > M\}\] is a finite collection of moduli.  For each
$m \in \cM$ we assume that a congruence class $a_m \bmod m$ is given.  The
uncovered set is
\[
 R = \bigcap_{m \in \cM} (a_m \bmod m)^c,
\]
which we show has a non-zero density.  In the general case it is convenient to
let
\[
Q = \LCM(m: m\in \cM),
\]
so that $R$ is a set defined modulo $Q$.

We take a sequence of thresholds $1 = P_{-1} < P_0 < P_1 < ...$ with $P_0 \geq
2$ and $P_i \to \infty$. Setting $v
= v_p =v_p(Q)$ for the multiplicity with which $p$ divides $Q$ we let
\[
 Q_{-1} = 1, \qquad \forall i \geq 0, \; Q_i = \prod_{p \leq P_i} p^{v}.
\]
Then $\cM_i = \{m \in \cM: m |Q_i\}$ is the collection of $P_i$-smooth
moduli in $\cM$.  The set $R$ is filtered in stages $R_{-1} \supset R_0 \supset
R_1 \supset ...$ by letting $R_{-1} = \zed$, and, for $i \geq 0$, 
\[
 R_i = \bigcap_{m \in \cM_i}  (a_m \bmod m)^c.
\]
Although $Q_i$ now depends in an essential way on the collection of moduli
$\cM$, our argument will, for a given $i$, treat the properties of $R_i$
uniformly for all distinct congruence systems having minimum modulus greater
than $M$.

\subsection{The initial stage}
 We are no longer able to assume that $Q_0 < M$ so that $\cM_0 =
\emptyset$, but we will assume that 
$M$ is sufficiently large so that $\cM_0$ is quite sparse.  Specifically, we
let $0 < \delta < 1$ be a parameter. We may estimate the density of the set 
\[
 R_0 = \bigcap_{m \in \cM_0} (a_m \bmod m)^c
\]
by applying the union bound
\begin{align*}
 |R_0 \bmod Q_0| &\leq Q_0 - \sum_{m  \in \cM_0} |(a_m \bmod m) \bmod Q_0|\\
 & = Q_0 \left(1 - \sum_{m \in \cM_0} \frac{1}{m}\right) \leq Q_0 \left(1 -
\sum_{\substack{m > M \\ p|m \Rightarrow p \leq P_0}}
\frac{1}{m}\right),
\end{align*}
and we make the condition that 
\begin{equation*}\tag{C0}
 \sum_{\substack{m > M\\ p|m \Rightarrow p \leq P_0}}
\frac{1}{m} < \delta.
\end{equation*}
This implies a bound for some bias statistics of $R_0$ as follows. 

Let
 $\ell_k(m)$ be the number of $k$-tuples of natural numbers having $\LCM$
$m$.  This is a multiplicative function, that is $\ell_k(mn) =
\ell_k(m)\ell_k(n)$ when $m$ and $n$ are co-prime, and it is given at prime
powers by
\[
 \ell_k(p^j) = (j+1)^k - j^k.
\] We define the $k$th bias  statistic at stage 0  to be
\begin{align*}
 \beta_k^k(0) &= \sum_{m |Q_0} \ell_k(m) \max_{b \bmod m} \frac{|R_0
\cap (b\bmod m) \bmod Q_0|}{|R_0 \bmod Q_0|}.
\end{align*}
Putting in the trivial bound $|R_0 \cap (b\bmod m) \bmod Q_0| \leq
\frac{Q_0}{m}$, we find
\[
 \beta_k^k(0) \leq 
 \frac{1}{1-\delta} \sum_{m
| Q_0} \frac{\ell_k(m)}{m} \leq \frac{1}{1-\delta} \prod_{p \leq
P_0}\left(\sum_{j=0}^\infty \frac{(j+1)^k - j^k}{p^j}\right).
\]
We now leave the initial stage.  We will return to choose $\delta$ and
$P_0$ at the end of the argument.

\subsection{The inductive loop}
In sieving stage $i+1$,  $i \geq 0$, 
we view $\zed/Q_{i+1}\zed$ as fibred over $\zed/Q_{i}\zed$, and  we consider
the set $R_{i+1}$ within individual fibres over
$R_{i}$. 

Introduce the set of `new moduli'
\[
 \cN_{i+1} = \{n : n |Q_{i+1},\; n > 1,\;  p|n \Rightarrow P_{i}<p \leq
P_{i+1}\},
\]
and notice that each $n \in \cN_{i+1}$ is coprime to $Q_i$.
Thus each modulus $m \in \cM_{i+1} \setminus \cM_{i}$ has a unique
factorization
as $m = m_0 n$ with $m_0|Q_{i}$ and $n \in \cN_{i+1}$.  Given $r \in R_{i}$
and $n \in \cN_{i+1}$ we set
\[
 A_{n,r} = (r \bmod Q_{i}) \cap \bigcup_{m_0 | Q_{i}, m_0 n \in \cM_{i+1}}
(a_{m_0 n} \bmod m_0 n).
\]
Then
\[
 (r \bmod Q_{i}) \cap R_{i+1} = (r \bmod Q_{i}) \cap \bigcap_{n \in \cN_{i+1}}
A_{n,r}^c.
\]

We wish to consider $R_{i+1}$ only in good fibres $(r \bmod Q_i)$ where the
sieve is well-behaved. A set
of properties that we would like good fibres to have is the following.

\begin{definition}
 Let $i \geq 0$ and let $\lambda \geq 0$ be a parameter.  We say that $r \in
\zed/Q_i\zed$ is $\lambda$-\emph{well-distributed} if $R_{i+1} \cap (r \bmod
Q_i)$ is non-empty, and if the fibre satisfies the uniformity property
that for each $n \in \cN_{i+1}$,
\begin{equation}\label{lambda_well_distributed} \max_{b \bmod n} \frac{|R_{i+1}
\cap
(b\bmod n) \cap
(r \bmod Q_i) \bmod Q_{i+1}| }{|R_{i+1}
\cap
(r \bmod Q_i) \bmod Q_{i+1}|} \leq \frac{e^{\lambda \omega(n)}}{n} .
\end{equation}
\end{definition}

An alternative, more technical characterization of good fibres is as follows.

\begin{definition}  Let $i \geq 0$ and let $\lambda \geq 0$ be a real
parameter.  We say that the fibre  $r \in R_i \bmod Q_i$ is
$\lambda$-\emph{good} if, for each $p \in (P_{i}, P_{i+1}]$, 
 \begin{equation}\label{lambda_good}
  \sum_{n \in \cN_{i+1}, p|n} \frac{|A_{n,r} \bmod n Q_{i}| e^{\lambda
\omega(n)}}{n} \leq 1-e^{-\lambda}.
 \end{equation}
\end{definition}
If each fibre in a set $S \subset R_i$ is $\lambda$-good, then we say that the
set $S$ is $\lambda$-good as well, similarly $\lambda$-well-distributed.

A basic observation of our proof  is that a $\lambda$-good fibre is
automatically $\lambda$-well-distributed.  

\begin{proposition}\label{good_fibre_prop}
 Let $i \geq 0$, $\lambda \geq 0$ and let $r \in \zed/Q_i\zed$ be
$\lambda$-good.  Then $r$ is $\lambda$-well-distributed.
\end{proposition}

The proof of this fact uses a relative form of the Lov\'{a}sz Local
Lemma.

\begin{lemma*}[Lov\'{a}sz Local Lemma, relative form]
 Let $\{A_u\}_{u \in V}$ be a finite collection of events in a probability
space. Let $D = (V, E)$ be
a directed graph, such that, for each $u \in V$, event $A_u$ is independent of
the sigma-algebra generated by the
events $\{A_v: (u,v) \not \in E\}$.  Suppose that there exist real numbers
$\{x_u\}_{u \in V},$ satisfying $0 \leq x_u < 1$, and for each $u \in V$, 
 \[
 \Prob(A_u) \leq x_u \prod_{(u,v) \in E} (1 - x_v).
 \]
 Then for any $\emptyset \neq U\subset V$
 \begin{equation}\label{local_lemma_strong}
  \Prob\left( \bigcap_{u\in V}A_u^c \right) \geq \Prob\left(
\bigcap_{u\in U} A_u^c\right) \cdot \prod_{v \in V\setminus U}(1 - x_v).
 \end{equation}
 In particular, taking $U$ to be a singleton,
 \begin{equation}\label{local_lemma_weak}
  \Prob\left( \bigcap_{u\in V} A_u^c\right) \geq \prod_{u \in V}
(1-x_u).
 \end{equation}

\end{lemma*}

\begin{remark}
 The conclusion (\ref{local_lemma_weak}) is the standard one, see \cite{ASE92}. 
The stronger conclusion (\ref{local_lemma_strong}) follows directly from the
proof.  For completeness, we show the argument in Appendix
\ref{LLL_appendix}, see also \cite{TV06}.
\end{remark}

The application of the Local Lemma to prove Proposition \ref{good_fibre_prop}
is as follows.  Write $F_r$
for the fibre $(r \bmod Q_{i}) \subset
\zed/Q_{i+1}\zed$ and make it a probability
space with the uniform measure $\Prob_r$.  The events are the collection
$\{A_{n,r}\}_{n
\in \cN_{i+1}}$.  Since $F_r$ contains $\frac{Q_{i+1}}{Q_{i}}$ elements, and
since $A_{n,r}$ is a set defined modulo $nQ_i$,
\[\Prob_r(A_{n,r}) = \frac{|A_{n,r} \bmod nQ_{i}|}{n}.\]  By first translating
by $-r$ and then dilating by $\frac{1}{Q_{i}}$ we map $F_r$ onto
$\zed/\frac{Q_{i+1}}{Q_{i}} \zed$.  For $n \in
\cN_{i+1}$, this map gives a bijection between progressions modulo $n Q_{i}$
constrained to $(r \bmod Q_{i})$, and unconstrained progressions modulo $n$ in
$\zed/\frac{Q_{i+1}}{Q_i}\zed$.
Applying this map, and then the Chinese Remainder Theorem, makes it clear that
$A_{n,r}$ is jointly independent of the $\sigma$-algebra generated by the events
 \[\{(b \bmod n') \cap (r \bmod Q_{i}): n' \in \cN_{i+1}, (n,n')=1\}.\]
In particular, a valid dependency graph with which to apply the Local Lemma has
edges between $n_1, n_2 \in \cN_{i+1}$ if and only if $n_1 \neq n_2$ and  $(n_1,
n_2) > 1$.

\begin{proof}[Proof of Proposition \ref{good_fibre_prop}]
 We first check that
 \[
  \forall n \in \cN_{i+1}, \qquad  x_n = e^{\lambda \omega(n)} \frac{|A_{n,r}
\bmod n Q_{i}|}{n}
 \]
is an admissible set of weights with which to apply the Local Lemma.

Since the fibre $r$ is $\lambda$-good, the bound in dilations
 condition (\ref{lambda_good}) gives 
 that for all $p \in
(P_{i}, P_{i+1}]$,
\[
 \sum_{n \in \cN_{i+1}: p|n} \frac{|A_{n,r} \bmod n Q_{i}| e^{\lambda
\omega(n)}}{n} \leq 1 - e^{-\lambda}.
\]
Dropping all but one term
in the sum, we see that for each $n \in \cN_{i+1}$, $1-x_n \geq e^{-\lambda}$. 
Thus, by convexity,
\[
1-x_n \geq \exp\left(\frac{-\lambda}{1-e^{-\lambda}} x_n\right).
 \]
 Therefore, for a given $n \in \cN_{i+1}$,
 \begin{align*}
  \prod_{n' \in \cN_{i+1}: (n,n')>1} (1-x_{n'}) &\geq \prod_{p|n} \prod_{n' \in
\cN_{i+1}: p|n'} (1-x_{n'})\\ & \geq \exp\left(\frac{-\lambda}{1-e^{-\lambda}}
\sum_{p|n} \sum_{n' \in \cN_{i+1}: p|n'} \frac{e^{\lambda \omega(n')} |A_{n',r}
\bmod n'Q_{i}|}{n'}\right)\\& \geq \exp\left(-\lambda \omega(n)\right).
 \end{align*}
 It follows
that
\[ 
 x_n \prod_{\substack{n' \in \cN_{i+1}: (n,n')>1\\ n' \neq n}} (1-x_{n'})\geq
x_n
\prod_{n' \in \cN_{i+1}: (n,n')>1} (1-x_{n'}) \geq \frac{|A_{n,r} \bmod
nQ_{i}|}{n}
\]
so that the Lov\'{a}sz criterion is satisfied.  It is then immediate that the
fibre itself is non-empty, since the product in the conclusion
(\ref{local_lemma_weak}) of the Local Lemma is non-zero.

For the uniformity property (\ref{lambda_well_distributed}),   let $n \in
\cN_{i+1}$ and let $b
\bmod n$ maximize
\[
 \frac{|R_{i+1} \cap (r \bmod Q_{i}) \cap (b\bmod n) \bmod Q_{i+1}|}{|R_{i+1}
\cap (r \bmod
Q_{i}) \bmod Q_{i+1}|} = \frac{\Prob_r \left( \left(\bigcap_{n' \in \cN_{i+1}}
A_{n',r}^c\right) \cap (b \bmod n)\right)}{\Prob_r \left( \bigcap_{n' \in
\cN_{i+1}} A_{n',r}^c \right)}.
\]
Dropping part of the intersection, the numerator is bounded above by 
\[
 \Prob_r \left( \left(\bigcap_{n' \in \cN_{i+1}, (n',n) = 1}
A_{n',r}^c\right) \cap (b \bmod n)\right) = \frac{1}{n} \Prob_r 
\left(\bigcap_{n' \in \cN_{i+1}, (n',n) = 1}
A_{n',r}^c\right).
\]
Now by the stronger conclusion (\ref{local_lemma_strong}) of the 
Local Lemma,
\[
 \Prob_r \left( \bigcap_{n' \in
\cN_{i+1}} A_{n',r}^c \right) \geq \Prob_r \left(\bigcap_{n' \in \cN_{i+1},
(n',n) = 1}
A_{n',r}^c\right) \prod_{n' \in \cN_{i+1}, (n',n)>1} (1 - x_{n'}).
\]
Since we checked above that 
\[
 \prod_{n' \in \cN_{i+1}, (n',n)>1} (1 - x_{n'}) \geq e^{-\lambda \omega(n)}
\]
it follows that
\[
 \frac{|R_{i+1} \cap (b\bmod n)\cap (r \bmod Q_{i}) \bmod Q_{i+1}|}{|R_{i+1}
\cap (r \bmod
Q_{i}) \bmod Q_{i+1}|} \leq \frac{1}{n} \prod_{n' \in \cN_{i+1},
(n',n)>1}(1-x_{n'})^{-1} \leq \frac{e^{\lambda \omega(n)}}{n},
\]
which is the condition of uniformity.

\end{proof}

Let $R_{-1}^* = \zed$, and for $i \geq 0$  let $R_i^*$ be the
$\lambda$-good fibres within $R_{i-1}^* \cap R_i$.   It remains to describe how
we may find good fibres above a large well-distributed set.

It will be convenient to reweight $\zed/Q_i\zed$ at each stage with a measure
$\mu_i$, supported on the set $R_{i-1}^* \cap R_i$.  The advantage of
using this measure is that it will balance the effect of the variation in
size of the various good fibres from previous stages, so that at stage $i+1$ we
 isolate the effects of sieving by moduli in $\cN_{i+1}$.  We define $\mu_i$
iteratively by setting 
\[\mu_{0}(r) = \left\{ \begin{array}{lll} \frac{1}{|R_0\bmod Q_0|}&&
 r \in R_0 \bmod Q_0\\ 0 && r \not \in R_0 \bmod Q_0\end{array}\right. .
 \]
 For $i
\geq 0$ and  for $r \in R_{i}^* \cap R_{i+1}\bmod Q_{i+1}$ we
reduce $r \bmod Q_i$ to determine $\mu_i(r)$, and set 
\[
\mu_{i+1}(r)=\left\{\begin{array}{lll} 
\frac{\mu_{i}(r \bmod Q_i )}{|R_{i+1} \cap (r \bmod Q_{i}) \bmod Q_{i+1}|} && r
\in R_i^* \cap R_{i+1} \bmod Q_{i+1}\\ 0 && r \not \in R_i^* \cap
R_{i+1} \bmod Q_{i+1}\end{array}\right. .
\]

Along with the measures $\mu_i$, we track a collection of bias statistics.
\begin{definition}
 Let $i \geq 0$ and $k \geq 1$.  The $k$th \emph{bias  statistic} of set
$R_{i-1}^* \cap R_i \subset \zed/Q_i\zed$ is defined by
\[
 \beta_k^k(i) =  \sum_{m|Q_i} \ell_k(m)
\max_{b \bmod m}\frac{\mu_i(R_{i-1}^* \cap R_i
\cap (b\bmod m))}{\mu_i(R_{i-1}^* \cap R_i)}.
\]
\end{definition}
Since we require $R_{-1}^* = \zed$ and since $\mu_0$ is uniform on $R_0$,
this agrees with our definition of the
bias statistics for $R_0$ given in the initial stage. These bias statistics
will be the main tool used to produce good fibres, a discussion which we
briefly postpone.  

The primary virtue of the measure $\mu_i$ is that it allows us to bound the
iterative growth of the bias statistics only in terms of the size of the
well-distributed set $R_i^*$ and its parameter of well-distribution, $\lambda$. 
Before demonstrating this, we record the notation
\[
 \pi_i^{\good} = \frac{\mu_i(R_i^*)}{\mu_i(R_{i-1}^* \cap R_i)}
\]
for the  fraction of good fibres in $R_{i-1}^* \cap R_i$, and we
record the
following simple lemma.
\begin{lemma}
Let $i \geq 0$.  For a fixed $r \in R_i^* \bmod Q_i$, the measure $\mu_{i+1}$
is constant on $R_{i+1} \cap (r \bmod Q_i)$.  The total mass of $\mu_{i+1}$
is given by
\[
 \mu_{i+1}(R_{i}^* \cap R_{i+1}) = \pi_i^{\good} \mu_i( R_{i-1}^* \cap R_{i}).
\]
\end{lemma}

\begin{proof}
The first observation is immediate from the definition.

The total mass is given by
\begin{align*}
 \mu_{i+1}(R_{i}^* &\cap R_{i+1})= \sum_{r \in R_{i}^* \cap R_{i+1} \bmod
Q_{i+1}} \mu_{i+1}(r)
\\&= \sum_{r_0 \in R_{i}^* \bmod Q_{i}} \mu_{i}(r_0) \sum_{r \in R_{i+1} \cap
(r_0 \bmod Q_{i})\bmod Q_{i+1}} \frac{1}{|R_{i+1} \cap (r_0 \bmod Q_{i}) \bmod
Q_{i+1}|}\\& = \sum_{r_0 \in R_{i}^*\bmod Q_i} \mu_{i}(r_0)\\& =
\pi_i^{\good}\mu_{i}(R_{i-1}^* \cap R_{i}).
\end{align*}
\end{proof}

The main proposition regarding the measures $\mu_i$ now is as follows.

\begin{proposition}
 Let $i \geq 0$ and $k \geq 1$ and suppose that $R_{i}^*$ is $
\lambda$-good.  We have
\[
 \beta_{k}^k(i+1) \leq  \frac{\beta_k^k(i)}{\pi_i^{\good}} \prod_{P_{i} <
p
\leq P_{i+1}}
\left(1 + e^\lambda \sum_{j=1}^{v_p} \frac{(j+1)^k - j^k}{p^j}
\right).
\]

\end{proposition}

\begin{proof}
Recall,
 \begin{equation}\label{bias_stats_def}
  \beta_k^k(i+1) = \sum_{m | Q_{i+1}} \ell_k(m) \max_{b \bmod m}
\frac{\mu_{i+1}(R_{i}^* \cap R_{i+1} \cap (b \bmod m))}{\mu_{i+1}(R_{i}^* \cap
R_{i+1})}.
 \end{equation}
Given $m |Q_{i+1}$ factor $m = m_0 n$ with $m_0 | Q_{i}$ and $n \in
\{1\}\cup\cN_{i+1}$. Let $b \bmod m$ maximize $\mu_{i+1}(R_{i}^* \cap R_{i+1}
\cap (b
\bmod m))$.  Fibring over $\zed/Q_{i}\zed$,  we have
\begin{align*}
 \mu_{i+1}(R_{i}^* \cap R_{i+1} &\cap (b \bmod m) )= \sum_{\substack{r_0
\in
R_{i}^* \bmod Q_{i}\\ r_0 \equiv b \bmod m_0}}\mu_{i+1}( (r_0
\bmod Q_{i}) \cap (b \bmod n))   \\ & = \sum_{\substack{r_0
\in
R_{i}^* \bmod Q_{i}\\ r_0 \equiv b \bmod m_0}} \mu_{i}(r_0) \frac{|R_{i+1}
\cap (b \bmod n) \cap (r_0 \bmod Q_{i}) \bmod Q_{i+1}|}{|R_{i+1} \cap (r_0 \bmod
Q_{i}) \bmod Q_{i+1}|}.
\end{align*}
Since the good set $R_{i}^*$ is $\lambda$-well-distributed, the last sum is
bounded by 
\begin{align*} \frac{e^{\lambda \omega(n)}}{n}
\sum_{\substack{r_0 \in
R_{i}^* \bmod Q_{i}\\ r_0 \equiv b \bmod m_0}} \mu_{i}(r_0) .
\end{align*}
Therefore, using the multiplicativity of $\ell_k(m)$, we find
\begin{align*}
 \beta_k^k(i+1) &\leq \sum_{n \in \{1\} \cup \cN_{i+1}} \frac{\ell_k(n)
e^{\lambda \omega(n)}}{n} \sum_{m_0 |Q_{i}} \ell_k(m_0)\max_{b \bmod m_0} 
\frac{\mu_{i}( R_{i}^* \cap (b \bmod m_0))}{\mu_{i+1}(R_{i}^* \cap R_{i+1})}.
\end{align*}
Since $\{1\} \cup \cN_{i+1}$ has the structure of a direct product, the sum over
$n$ factors as the product of the proposition.  Meanwhile, using $R_{i}^*
\subset R_{i-1}^* \cap R_{i}$ and $\mu_{i+1}(R_{i}^* \cap R_{i+1}) =
\pi_i^{\good} \mu_i(R_{i-1}^* \cap R_{i})$, we  bound the sum over $m_0$ by
\begin{align*}
&\sum_{m_0 |Q_{i}} \ell_k(m_0)\max_{b \bmod m_0} 
\frac{\mu_{i}( R_{i}^* \cap (b \bmod m_0))}{\mu_{i+1}(R_{i}^* \cap R_{i+1})}
\\&\leq \frac{1}{\pi_i^{\good}} \sum_{m_0 |Q_{i}} \ell_k(m_0)\max_{b \bmod
m_0} 
\frac{\mu_{i}( R_{i-1}^*\cap R_{i} \cap (b \bmod m_0))}{\mu_{i}(R_{i-1}^*
\cap R_{i})} = \frac{\beta_k^k(i)}{\pi_i^{\good}}.
\end{align*}

\end{proof}

It  remains to demonstrate the utility of the bias statistics for
generating good fibres. 
For $n \in \cN_{i+1}$, $k \geq 1$ and $R_{i-1}^* \cap R_i$ defined modulo
$Q_i$, define the $k$th moment of $|A_{n,r} \bmod nQ_i|$ to be 
\[
 M_k^k(i,n) = \frac{1}{\mu_i(R_{i-1}^* \cap R_{i})} \sum_{r \in
R_{i-1}^* \cap R_{i} \bmod Q_i}\mu_i(r) |A_{n,r} \bmod n Q_{i}|^k.
\]
The bias statistics control these moments.
\begin{lemma}\label{moment_lemma}
Let $i \geq 0$ and let $n \in \cN_{i+1}$.  
We have $ M_k(i,n) \leq \beta_k(i).$
\end{lemma}

\begin{proof}
 Recall that \[A_{n,r} = (r \bmod Q_i) \cap \left(\bigcup_{m_0 |Q_i, m_0 n \in
\cM} (a_{m_0 n}
\bmod m_0 n)
\right).\]  A given congruence $(a_{m_0 n} \bmod m_0 n)$
intersects $r \bmod Q_i$ if and only if $r\equiv a_{m_0 n} \bmod m_0$.  If it
does
intersect, it does so in a single residue class modulo $ nQ_i$.  Thus, the union
bound gives
\[
 |A_{n,r} \bmod n Q_i| \leq \sum_{m_0 |Q_i} \one\{r \equiv a_{m_0n} \bmod m_0\}.
\]
It follows that, considering $R_{i-1}^* \cap R_i$ as a subset of $
\zed/Q_i\zed$,
\begin{align*}
 M_k^k(i,n) &\leq \frac{1}{\mu_i(R_{i-1}^* \cap R_i)} \sum_{r \in
R_{i-1}^* \cap R_i }\mu_i(r) \sum_{m_1, ..., m_k |Q_i} \one\{\forall 1 \leq j
\leq k, \; r \equiv a_{m_j n} \bmod m_j\}
\\ & = \frac{1}{\mu_i(R_{i-1}^* \cap R_i)}\sum_{m_1, ..., m_k |Q_i} \sum_{r \in
R_{i-1}^* \cap R_i }\mu_i(r)\one\{\forall 1
\leq j
\leq k, \; r \equiv a_{m_j n} \bmod m_j\} .
\end{align*}
The inner condition restricts $r$ to at most one class modulo the LCM of
$m_1, ..., m_k$.  Grouping $m_1, ..., m_k$ according to their LCM, and
writing $\ell_k(m)$ for the number of ways in which $m$ is the LCM of a
$k$-tuple of natural numbers, we find
\begin{align*}M_k^k(i,n) \leq \frac{1}{\mu_i(R_{i-1}^* \cap R_i )} \sum_{m
|Q_i}
\ell_k(m) \max_{b \bmod m} \mu_i(R_{i-1}^* \cap R_i \cap (b \bmod m))
=
\beta_k^k(i).\end{align*}
\end{proof}

Since the above estimate is uniform in $n$, we have convexity-type control over
mixtures of the sizes
 $\{|A_{n,r} \bmod nQ_i|\}_{n \in \cN_{i+1}}$.
\begin{lemma}\label{convexity_lemma}
 Let $i \geq 0$ and $k \geq 1$.  Let $\{w_n\}_{n \in \cN_{i+1}}$ be a set of
non-negative
weights, not all zero.  Then for all $B> 0$ and any $k \geq 1$
\begin{align*}
 \frac{1}{\mu_i(R_{i-1}^* \cap R_i)}\mu_i\left(r \in R_{i-1}^* \cap
R_i:\sum_{n \in
\cN_{i+1}} w_n |A_{n,r}
\bmod nQ_i| > B \right)\leq \frac{ \beta_k^k(i)}{B^k}\left(\sum_{n \in
\cN_{i+1}}
w_n\right)^k.&
\end{align*}
\end{lemma}
\begin{proof}
 Set $w_n' = \frac{w_n}{\sum_{\tilde{n}} w_{\tilde{n}}}$, which is a probability
measure on
$\cN_{i+1}$. Convexity gives
\[
 \left(\sum_{n \in
\cN_{i+1}} w_n' |A_{n,r}
\bmod nQ_i| \right)^k \leq \sum_{n \in
\cN_{i+1}} w_n' |A_{n,r}
\bmod nQ_i|^k, 
\]
so that 
\[
 \frac{1}{\mu_i(R_{i-1}^* \cap R_i)} \sum_{r \in R_{i-1}^* \cap R_i }\mu_i(r)
\left(\sum_{n
\in
\cN_{i+1}} w_n' |A_{n,r}
\bmod nQ_i| \right)^k \leq \sum_{n \in
\cN_{i+1}} w_n' M_k^k(i,n) \leq \beta_k^k(i).
\]
The result now follows from Markov's inequality.
\end{proof}

We now complete  our argument by using the bias statistics to
guarantee the existence of good fibres. 

For a given $p \in (P_{i}, P_{i+1}]$, the  dilation condition of good fibres
(\ref{lambda_good}) at $p$ is
the
statement that 
\[
 \sum_{n \in \cN_{i+1}, p|n} \frac{|A_{n,r} \bmod n Q_{i}| e^{\lambda
\omega(n)}}{n} \leq 1-e^{-\lambda}.
\]
By applying the convexity lemma, Lemma \ref{convexity_lemma}, with weights
\[
 w_n = \one_{p|n} \frac{e^{\lambda \omega(n)}}{n},
\]
we find that the fraction of fibres failing this condition is bounded by
\[
 \min_k \frac{\beta_k^k(i)}{(1- e^{-\lambda})^k} \left(\sum_{n \in \cN_{i+1}:
p|n} \frac{e^{\lambda \omega(n)}}{n}\right)^k.
\]
Since 
\[
 \sum_{n \in \cN_{i+1}, p|n} \frac{e^{\lambda \omega(n)}}{n} \leq
\frac{e^{\lambda}}{p-1} \sum_{n \in \{1\}\cup \cN_{i+1}} \frac{e^{\lambda
\omega(n)}}{n} \leq \frac{e^{\lambda}}{p-1} \prod_{P_{i}<p' \leq P_{i+1}}\left(1
+
\frac{e^{\lambda}}{p'-1}\right),
\]
making a union bound, we find that the total fraction of fibres failing
some dilation condition is bounded by 
\[
 \min_k \beta_k^k(i)\frac{e^{k \lambda}}{(1-e^{-\lambda})^k
}\left(\prod_{P_{i}<p \leq P_{i+1}}\left(1 +
\frac{e^{\lambda}}{p-1}\right)
\right)^k \sum_{P_{i}< p \leq P_{i+1}} \frac{1}{(p-1)^k}.
\]
For a value $0 < \pi^{\good}<1$, we make the constraint that this quantity is
bounded by $1-\pi^{\good}$, that
is
\begin{equation*}\tag{C1}
\frac{e^{\lambda}}{1-e^{-\lambda}} \prod_{P_{i}<p \leq P_{i+1}}\left(1 +
\frac{e^{\lambda}}{p-1}\right) \leq  
\max_k
\frac{(1-\pi^{\good})^{\frac{1}{k}}}{\beta_k(i)} \left(\sum_{P_{i} < p
\leq
P_{i+1}} \frac{1}{(p-1)^k} \right)^{-\frac{1}{k}},
\end{equation*}
which guarantees that, with respect to $\mu_i$, the fraction of good fibres
in $R_{i-1}^* \cap R_i$ is at least $\pi^{\good}$.

\subsection{Proof of Theorem \ref{main_result}}
The iterative stage of our argument is summarized in the following technical
theorem.

\begin{main_theorem}\label{technical_theorem}
Let $i \geq 0$ and let $0<\pi^{\good} <1$.  Let the set $R_{i-1}^*$ exist such
that $R_{i-1}^* \cap
R_{i}$ is non-empty, with associated measure $\mu_i$ and bias statistics
$\beta_k(i)$, $k = 1,
2, 3, ...$. Suppose that $\lambda>0$ and $P_{i+1} > P_{i}$ satisfy the 
constraint
\begin{equation*}\tag{C1}
 \prod_{P_{i}<p \leq P_{i+1}}\left(1 +
\frac{e^{\lambda}}{p-1}\right) \leq
\frac{1-e^{-\lambda}}{e^{\lambda}}\max_k
\frac{(1-\pi^{\good})^{\frac{1}{k}}}{\beta_k(i)} \left(\sum_{P_{i} < p
\leq
P_{i+1}} \frac{1}{(p-1)^k} \right)^{-\frac{1}{k}}.
\end{equation*}
Then there exists $R_i^* \subset R_{i-1}^* \cap R_i$
defined modulo $Q_{i}$ with $\frac{\mu_i(R_i^*)}{\mu_i(R_{i-1}^* \cap R_i)}
\geq \pi^{\good}$. The density of $R_{i+1}$ in each fibre above $R_{i}^*$
is positive, and the associated bias statistics $\beta_k(i+1)$ of $R_i^* \cap
R_{i+1}$ with respect to $\mu_{i+1}$ satisfy
\[
 \beta_k^k(i+1) \leq \frac{\beta_k^k(i)}{\pi^{\good}} \prod_{P_{i} < p \leq
P_{i+1}}\left(1 +
e^{\lambda} \sum_{j=1}^{v_p}
\frac{(j+1)^k - j^k}{p^j}\right), \qquad k = 1, 2, ....
\]
\end{main_theorem}

We now make specific choices for our parameters and prove Theorem
\ref{main_result}. 

\begin{proof}[Proof of Theorem \ref{main_result}]
Set $M =  10^{16}$ as in Theorem \ref{main_result}. For $i \geq 0$,
let $P_i =
e^{11 + i}$. Set $e^\lambda = 2$, $\pi^{\good} =
\frac{1}{2}$. It will suffice to check that the density of the set $R_0$ is
positive, and that the constraint (C1) of Theorem \ref{technical_theorem} is met
for every $i \geq 0$.

By Rankin's trick, for any $\sigma > 0$,
\[
 \sum_{\substack{m >M\\ p|m \Rightarrow p \leq P_0}}
\frac{1}{m} \leq M^{-\sigma} \sum_{m : p|m \Rightarrow p \leq P_0}
\frac{1}{m^{1-\sigma}} = M^{-\sigma} \prod_{p \leq P_0} \left(1 -
\frac{1}{p^{1-\sigma}}\right)^{-1}.
\]
Choosing $\sigma = 0.19$, we verify in Pari-GP \cite{PariGP} that the right hand
side is less than
$0.859$, so that $R_0$ is non-empty, and, in particular, $\delta = 0.86$ in
the
initial stage is permissible.

We will argue throughout with the 3rd bias statistic.  We calculate
\begin{align*}
\beta_3(0) & \leq \left((1-\delta)^{-1} \prod_{p \leq P_0}
\left(\sum_{j=0}^\infty \frac{3j^2 + 3j + 1}{p^j}\right)\right)^{\frac{1}{3}} <
731.8.
\end{align*}

We use the following explicit estimates, which are verified in Appendix A.  
For all $n\geq 11$,
\begin{align*}
 &\prod_{e^{n}<p \leq e^{n+1}} \left(1 + \frac{2}{p-1}\right) < 1.2.
\\& \prod_{e^{n}<p \leq e^{n+1}}\left(1 + 2 \sum_{j=1}^\infty \frac{(j+1)^3 -
j^3}{p^j}\right)  < 3.4.
\\& \left(\sum_{e^{n}<p \leq e^{n+1}} \frac{1}{(p-1)^3}\right)^{-\frac{1}{3}} 
> (2n e^{2n})^{\frac{1}{3}}.
\end{align*}
Thus the constraint (C1) is satisfied at $i = 0$, since
\[
 \prod_{e^{11}<p \leq e^{12}} \left(1 + \frac{2}{p-1}\right) < 1.2
 < \frac{(1-0.5)^{\frac{1}{3}}}{4} \frac{1}{731.8}  \left(\sum_{e^{11} < p
\leq e^{12}} \frac{1}{(p-1)^3}\right)^{\frac{-1}{3}}.
\]
The constraint holds for all $i$, since the growth of the bias
statistics guarantees that for $i \geq 0$,
\[
 \frac{\beta_3(i+1) }{\beta_3(i)}< \left(\frac{
3.4}{0.5}\right)^{\frac{1}{3}} < e^{\frac{2}{3}},
\]
which is less than the  growth of $\left((22 + 2i)e^{22 +
2i}\right)^{\frac{1}{3}}$ from $i$ to $i+1$.
\end{proof}

\newpage
\appendix

\section{Explicit estimates with primes}\label{prime_sum_appendix}

A standard reference for explicit prime sum estimates is \cite{RS75}.  Slightly
stronger estimates are now known, (see e.g. \cite{FK13}) but the following will
suffice for our purpose.

\begin{theorem}[\cite{RS75} Corollary 2]
 Let $\theta(x) = \sum_{p \leq x} \log p$.  For $x\geq 678407$ we have
 \begin{equation}\label{theta_error}
  |\theta(x) - x| < \frac{x}{40 \log x}.
 \end{equation}
\end{theorem}
We now check the explicit estimates used in the proof of Theorem
\ref{main_result}.
\begin{lemma}\label{explicit_lemma}
For any $n \geq 11$ 
 \begin{align*}
  &\prod_{e^n < p \leq e^{n+1}} \left(1 + \frac{2}{p-1}\right)< 1.2 \\
  &\prod_{e^n < p \leq e^{n+1}} \left(1 + 2 \sum_{j=1}^\infty
\frac{(j+1)^3 - j^3}{p^j}\right) < 3.4\\
  &\sum_{e^n  < p \leq e^{n+1}} \frac{1}{(p-1)^3} <  \frac{1}{2n e^{2n}}.
 \end{align*}
 \end{lemma}

\begin{proof}
Using Pari-GP \cite{PariGP} we verified these estimates numerically for $n= 11,
12, 13$. For $n > 13$ they
follow by partial summation against (\ref{theta_error}).  For the first,
\[
 \log \prod_{e^n < p \leq e^{n+1}} \left(1 + \frac{2}{p-1}\right) \leq 2
\sum_{e^n< p \leq e^{n+1}}\frac{1}{p-1} \leq
\frac{2}{1-e^{-n}}\int_{e^n}^{e^{n+1}}
\frac{d\theta(x)}{x\log x}.
\]
Write $d\theta(x) = dx + d(\theta(x) - x)$.  Integrating the second term by
parts, we obtain  
\begin{align*}
\int_{e^n}^{e^{n+1}}
\frac{d\theta(x)}{x\log x} &\leq \log \frac{n+1}{n} +
\frac{|\theta(e^{n+1})-e^{n+1}|}{(n+1)e^{n+1}} + \frac{|\theta(e^n) -
e^n|}{ne^n}\\&\qquad\qquad\qquad + 
\int_{e^n}^{e^{n+1}} \frac{|\theta(x)-x|}{x^2} \left(\frac{1}{\log x} +
\frac{1}{(\log x)^2}\right) dx\\
& \leq  \log \frac{15}{14} +
\frac{1}{40 \cdot 15^2} + \frac{1}{40 \cdot 14^2 } + \frac{2}{40 \cdot 14} \log
\frac{15}{14} < 0.0695
\end{align*}
so that
\[
 \frac{2}{1-e^{-14}} \int_{e^n}^{e^{n+1}} \frac{d\theta(x)}{x\log x}  <0.14<
\log 1.2.
\]

For the second,
\begin{align*}
 \log \prod_{e^n < p \leq e^{n+1}} \left(1 + 2 \sum_{j=1}^\infty
\frac{(j+1)^3 - j^3}{p^j}\right) &\leq 2 \sum_{e^n < p \leq e^{n+1}}
\sum_{j=1}^\infty
\frac{(j+1)^3 - j^3}{p^j} \\& \leq 14 \sum_{e^n < p \leq e^{n+1}}
\frac{1}{p-3}\\& \leq \frac{14}{1-3 e^{-14}} \sum_{e^n < p \leq e^{n+1}}
\frac{1}{p} \\ & < \frac{14}{1-3 e^{-14}} \cdot 0.07 < 1 <
\log(3.4).
\end{align*}

For the third, proceed as for the first,
\begin{align*}
 \sum_{e^n < p \leq e^{n+1}} &\frac{1}{(p-1)^3} \leq \frac{1}{n(1-e^{-n})^3}
\left( \int_{e^n}^{e^{n+1}} \frac{dx}{x^3} + \int_{e^n}^{e^{n+1}}
\frac{d(\theta(x) - x)}{x^3 }\right)
\\ & \leq \frac{1}{(1-e^{-n})^3} \left[\frac{1-e^{-2}}{2n e^{2n}} +
\frac{1}{40 n^2 e^{2n}} + \frac{1}{40 n(n+1) e^{2(n+1)}} + \frac{3}{40 n^2}
\int_{e^n}^{e^{n+1}} \frac{dx}{x^3 }\right]
\\& \leq \frac{1}{2ne^{2n}} \frac{1}{(1-e^{-14})^3} \left[1 - e^{-2} +
\frac{1}{20 \cdot 14} + \frac{1}{20 e^2 \cdot 15} + \frac{3}{40 \cdot
14}\right]\\
& < \frac{0.88}{2n e^{2n}}.
\end{align*}

\end{proof}

\section{The relative Lov\'{a}sz Local Lemma}\label{LLL_appendix}

For completeness, and for the reader's convenience, we record a proof of the
relative form of the Lov\'{a}sz Local Lemma used in our argument.  We emphasize
that the proof is the standard one, see for instance \cite{ASE92} pp.
54--55, although the conclusion that we need is not typically recorded.  

Recall the statement of the lemma.

\begin{lemma*}[Lov\'{a}sz Local Lemma, relative form]
 Let $\{A_u\}_{u \in V}$ be a finite collection of events in a probability
space. Let $D = (V, E)$ be
a directed graph, such that, for each $u \in V$, event $A_u$ is independent of
the sigma-algebra generated by the
events $\{A_v: (u,v) \not \in E\}$.  Suppose that there exist real numbers
$\{x_u\}_{u \in V},$ satisfying $0 \leq x_u < 1$, and for each $u \in V$, 
 \[
 \Prob(A_u) \leq x_u \prod_{(u,v) \in E} (1 - x_v).
 \]
 Then for any $\emptyset \neq U\subset V$
 \begin{equation}\label{appendix_local_lemma_strong}
  \Prob\left( \bigcap_{u\in V}A_u^c \right) \geq \Prob\left(
\bigcap_{u\in U} A_u^c\right) \cdot \prod_{v \in V\setminus U}(1 - x_v).
 \end{equation}
 In particular, taking $U$ to be a singleton,
 \begin{equation}\label{appendix_local_lemma_weak}
  \Prob\left( \bigcap_{u\in V} A_u^c\right) \geq \prod_{u \in V}
(1-x_u).
 \end{equation}

\end{lemma*}

\begin{proof}
By assigning an ordering to $V$, identify it with the set $\{1,2, ...,n\}$ for
some $n$.  Assume that in this ordering $U$ is identified with $\{1, 2, ...,
m\}$ for some $m$.

The following is to be shown by induction.  For $k = 1, 2, ..., n$,
\begin{enumerate}
 \item For any $S \subset \{1, ..., n\}$, $|S| = k-1$, and for any $1 \leq i
\leq n$, $i \not \in S$ we have
 \[ \Prob\left(A_i \; \bigg| \; \bigcap_{j \in S} A_j^c\right) \leq x_i\]
 \item For any $S \subset \{1, ..., n\}$, $|S| = k$ we have
 \[ \Prob\left(\bigcap_{j \in S} A_j^c\right) \geq \prod_{j \in S} (1 - x_j).\]
\end{enumerate}
Obviously (\ref{appendix_local_lemma_weak}) is the second item when $k =
n$.  The conclusion (\ref{appendix_local_lemma_strong}) is also easily deduced: 
\[
 \Prob\left(\bigcap_{i=1}^n A_i^c\right) = \Prob\left(\bigcap_{i =
1}^m A_i^c\right)\cdot  \prod_{j = m+1}^n \Prob\left(A_{j}^c
\;\bigg |\; \bigcap_{i = 1}^{j-1} A_i^c\right) \geq
\Prob\left(\bigcap_{i = 1}^m A_i^c\right) \cdot \prod_{j = m+1}^n (1
- x_j).
\]

When $k = 1$, the conditional statement is to be interpreted as if there is no
conditioning, and both statements are then obvious.  

To induce, let $1 < k \leq n$ and
assume the truth of both statements for any $1 \leq k' < k$.  We first prove
statement 1 in case $k$.  Note that by the case $k-1$ of statement 2, the
conditional probability in 1 is well defined. Let $S_1 = \{j \in S: (i,j) \in
E\}$ and let $S_2 = S \setminus S_1$. We may obviously assume that $S_1= \{j_1
<
j_2 < ... < j_r\}$ is non-empty, since otherwise the result is immediate by
independence. We have
\[
 \Prob\left(A_i \; \bigg| \; \bigcap_{j \in S} A_j^c\right) = \frac{
\Prob\left(A_i
\cap \bigcap_{j \in S_1} A_j^c \; \bigg| \; \bigcap_{j \in S_2}
A_j^c\right)}{\Prob\left(\bigcap_{j \in S_1} A_j^c \; \bigg| \; \bigcap_{j \in
S_2}
A_j^c\right)}.
\]
For the denominator we have the lower bound
\begin{align*}
 \Prob\left(A_{j_1}^c \; \bigg| \; \bigcap_{j \in S_2} A_j^c\right)\cdot
 \Prob\left(A_{j_2}^c \; \bigg| \; A_{j_1}^c \cap \bigcap_{j \in S_2} A_j^c
\right)\cdot ... \cdot \Prob\left(A_{j_r}^c \; \bigg| \; \bigcap_{\ell
=1}^{r-1}
A_{j_\ell}^c \cap \bigcap_{j \in S_2} A_j^c\right)&\\ \geq \prod_{\ell = 1}^r (1
-
x_{j_\ell}),&
\end{align*}
by applying 1 of the inductive assumption in cases $k' < k$.

For the numerator we have the upper bound
\[
 \Prob\left(A_i \cap \bigcap_{j \in S_1} A_{j}^c \; \bigg| \; \bigcap_{j 
 \in S_2} A_j^c\right) \leq \Prob\left(A_i \; \bigg| \; \bigcap_{j \in S_2}
A_j^c
\right) = \Prob(A_i) \leq x_i \prod_{j: (i,j) \in E} (1-x_j).
\]
Combined, these two bounds prove 1 in case $k$.

To prove 2 in case $k$, let $S = \{j_1< j_2< ...< j_r\}$ and observe
\[
 \Prob\left(\bigcap_{j \in S} A_j^c\right) = \prod_{\ell = 1}^r
\Prob\left(A_\ell^c \; \bigg| \; \bigcap_{1 \leq m < \ell} A_m^c\right) \geq
\prod_{\ell = 1}^r (1 - x_\ell),
\]
which uses 1 in case $k$.
\end{proof}

\end{document}